\documentclass[reqno, 11pt]{amsart}
\usepackage{graphicx, enumerate, enumitem, amsmath, amssymb,amsthm,hyperref,amscd,xcolor}
\usepackage{mathabx,manfnt,braket,mathrsfs}
\usepackage[makeroom]{cancel}
\usepackage[utf8]{inputenc}

\newtheorem*{thm*}{Theorem}
\newtheorem*{cor*}{Corollary}

\theoremstyle{definition}

\theoremstyle{remark}

\newcommand{\norm}[1]{\left\Vert#1\right\Vert}
\newcommand{\abs}[1]{\left\vert#1\right\vert}
\newcommand{\cx}{{\mathbb{C}}}
\newcommand{\rl}{{\mathbb{R}}}
\newcommand{\D}{\mathbb{D}}

\newcommand{\Z}{\mathbb{Z}}
\newcommand{\N}{\mathbb{N}}

\newcommand{\ol}{\overline}

\newcommand{\wt}{\widetilde}

\title{On an observation of Sibony}
\author{Debraj Chakrabarti}
\address{Department of Mathematics, Central Michigan University, Mount Pleasant, MI 48859, U.S.A.}
\email{chakr2d@cmich.edu}

\thanks{The author was  partially supported by a National Science Foundation grant (DMS-1600371), and  by a collaboration grant from the Simons Foundation (\# 316632).}

\begin{document}
\maketitle
\begin{abstract}
It is shown that if the boundary of a Reinhardt domain  in $\mathbb{C}^n$ contains the origin,
then the origin has a neighborhood to which each holomorphic function on the domain which is infinitely many times differentiable up to the boundary extends 
holomorphically. 
\end{abstract}
For a domain $\Omega$ in $\cx^n$, let  $\mathcal{C}^\infty(\ol{\Omega})$ denote the space of those smooth  functions on $\Omega$  whose partial derivatives of all orders extend continuously to  the closure $\ol{\Omega}$, and let $\mathcal{O}(\Omega)$ be the space of holomorphic functions on $\Omega$.
In this note we prove the following result:

\begin{thm*}\label{thm-sibony} Let $\Omega\subset\cx^n$ be a Reinhardt domain such that the origin is a boundary
point of $\Omega$. Then there is a neighborhood of the origin such that  each  function in  $\mathcal{C}^\infty(\ol{\Omega})\cap\mathcal{O}(\Omega)$ extends holomorphically to this neighborhood.
\end{thm*}

The special case of this result when $\Omega$ is the ``Hartogs Triangle''  $H=\{\abs{z_1}< \abs{z_2}<1\}\subset \cx^2$ was noted by Sibony in \cite[p. 220]{sibony1975}, where the neighborhood of the origin to which functions extend is the bidisc.  Sibony's result constitutes a  refinement of the classical fact that 
the {\em Nebenhülle} of the Hartogs triangle is the bidisc (see \cite{behnke1933}), i.e. the bidisc is the largest open set contained in each Stein neighborhood of the closure  $\ol{H}$. Consequently, each function holomorphic in a neighborhood of $\ol{H}$ extends to the bidisc.  As is well-known, it is possible for a smoothly bounded pseudoconvex  domain  to have nontrivial Nebenhülle  (see \cite{worm}). On the other hand, each smoothly bounded pseudoconvex domain $\Omega$
is a {\em $\mathcal{C}^\infty$-domain of holomorphy}, i.e. $\Omega$
admits a function in $\mathcal{C}^\infty(\ol{\Omega})\cap \mathcal{O}(\Omega)$ which does not extend holomorphically past any boundary point (see \cite{catlin, hakimsibony}).  Therefore, the Hartogs triangle, which is not a $\mathcal{C}^\infty$-domain of holomorphy, displays a behavior which is specific to  non-smooth domains. The problem of geometrically characterizing the Reinhardt $\mathcal{C}^\infty$-domains of holomorphy among the Reinhardt domains of holomorphy was solved in \cite[Proposition~6]{jarnickipflug1997} (see also \cite[Section~3.5]{jarpflubook}).  The result of \cite{jarnickipflug1997} implies that a Reinhardt domain of holomorphy with the origin on the boundary is not a $\mathcal{C}^\infty$-domain of holomorphy, a fact which also follows from our theorem. However, our result  not only identifies certain Reinhardt domains $\Omega$ (whether or not domains of holomorphy) as not being $\mathcal{C}^\infty$-domains of holomorphy, but also gives a simple 
and concrete holomorphic extension of each function in $\mathcal{C}^\infty(\ol{\Omega})\cap \mathcal{O}(\Omega)$ to an explicit fixed larger domain containing the origin (see corollary below after the proof of the theorem).

Of course,  for a domain $\Omega$,  the existence of a Nebenhülle or a $\mathcal{C}^\infty$-envelope of holomorphy  (smallest open set to which all functions in $\mathcal{C}^\infty(\overline{\Omega})\cap \mathcal{O}(\Omega)$ extend holomorphically)
  is  most interesting when $\Omega$ is pseudoconvex, i.e., when there is no point outside $\Omega$ to which each function holomorphic on $\Omega$ extends. Using our theorem, it is easy to give  examples of (nonsmooth)
pseudoconvex domains 
with nontrivial $\mathcal{C}^\infty$-envelopes analogous to the Hartogs triangle: if $\alpha=(\alpha_1,\dots, \alpha_n)\in \Z^n$ is a multi-index such that   $\alpha\not \in \N^n$ and $-\alpha\not \in \N^n$, then let
\[ H(\alpha)= \{ z\in \D^n| \abs{z^\alpha} <1\},\]
where $\D^n$ is the unit polydisc.   From the theorem,  $H(\alpha)$ has a nontrivial $\mathcal{C}^\infty$-envelope of holomorphy, and noting the convexity of the image of the map $z\mapsto (\log\abs{z_1},\dots, \log\abs{z_n})$ we see that $H(\alpha)$ is pseudoconvex (see 
 \cite[Section 3.8]{range}).   For $n=2$, these domains are precisely the ``fat'' and ``thin'' generalized Hartogs triangles of rational exponent, which have been studied extensively recently (see \cite{ChaZey16, EdhMcN16,Edh16})  and shown to have unexpected properties as far as the $L^p$ regularity of the 
Bergman projection is concerned.  
\begin{proof} We will use the usual multi-index notation in dealing with functions of several variables. Also, we will assume that $\Omega$ is bounded. This is no loss of generality, since we can always replace $\Omega$ by its intersection with a polydisc and prove the result for this bounded domain.

Let $f\in \mathcal{C}^\infty(\ol{\Omega})\cap \mathcal{O}(\Omega).$
Since $f$ is holomorphic on the Reinhardt 
domain $\Omega$, there is a Laurent expansion
\begin{equation}\label{eq-laurent}
f(z)= \sum_{\alpha\in \Z^n} c_\alpha z^\alpha
\end{equation}
which is uniformly and absolutely convergent on compact subsets of $\Omega$ (see, e.g., \cite{range}). Let $Z=\{w\in \cx^n | w_j=0 \text{ for some } \, j\}$. The coefficients $c_\alpha\in \cx$ are represented by the well-known Cauchy formula: if $w\in \Omega\setminus Z$,  we have the $n$-fold repeated line-integral representation
\[ c_\alpha = \frac{1}{(2\pi i)^n }\int_{\abs{z_1}=\abs{w_1}} \cdots \int_{\abs{z_n}=\abs{w_n}} \frac{f(z)}{z^\alpha}\cdot \frac{dz_n}{z_n}\cdots \frac{dz_1}{z_1}.\]
Parametrize the contours by $z_j = w_j e^{i\theta_j}$, where $0\leq \theta_j \leq 2\pi$.   Notice that then we have
\[ dz_j = w_j\cdot i e^{i\theta_j} d\theta_j=i z_j d\theta_j,\]
so that
\[ c_\alpha =  \frac{1}{(2\pi i)^n }\int_{\theta_1=0}^{2\pi} \cdots \int_{\theta_n=0}^{2\pi} \frac{f(w_1e^{i\theta_1}, \dots, w_n e^{i\theta_n})}{w^\alpha\exp\left(i(\alpha_1\theta_1+\dots+\alpha_n\theta_n)\right)} (i d\theta_n) \dots (id\theta_1).\] 
Therefore, using an obvious ``vector-like'' notation, we have for each $w\in \Omega\setminus Z$ that
\[ c_\alpha w^\alpha =  \frac{1}{(2\pi)^n }\int_{\theta_1=0}^{2\pi} \cdots \int_{\theta_n=0}^{2\pi} \frac{f(w\cdot e^{i\theta})}{\exp\left(i\langle \alpha, \theta \rangle \right)} d\theta_n \dots d\theta_1.\] 
Write the multi-index $\alpha= \beta-\gamma$, where $\beta_j=\max(\alpha_j,0)$ and  $\gamma_j=\max(-\alpha_j,0)$. Then $\beta, \gamma\in \N^n$ and we can rewrite
\[ c_{\beta-\gamma}\cdot \frac{w^\beta}{w^\gamma} =  \frac{1}{(2\pi)^n }\int_{\theta_1=0}^{2\pi} \cdots \int_{\theta_n=0}^{2\pi} \frac{f(w\cdot e^{i\theta})}{\exp\left(i\langle \beta-\gamma, \theta \rangle \right)} d\theta_n \dots d\theta_1.\] 
Apply the differential operator $\left(\frac{\partial}{\partial w}\right)^\beta$ to both sides, which gives us
\begin{equation} \label{eq-afterdiff}  \frac{c_{\beta-\gamma}\cdot\beta!}{w^\gamma} = 
 \frac{1}{(2\pi)^n }\int_{\theta_1=0}^{2\pi} \cdots \int_{\theta_n=0}^{2\pi} {\exp\left(i\langle \gamma, \theta \rangle \right)}\,{\frac{\partial^\beta f}{\partial w^\beta}(w\cdot e^{i\theta})} \,d\theta_n \dots d\theta_1. \nonumber \end{equation}
 Taking absolute values, and doing a simple sup norm estimate, we see that
 \begin{align*} \abs{  \frac{c_{\beta-\gamma}\cdot\beta!}{w^\gamma} }&= \abs{\frac{1}{(2\pi)^n }\int_{\theta_1=0}^{2\pi} \cdots \int_{\theta_n=0}^{2\pi} {\exp\left(i\langle \gamma, \theta \rangle \right)}\,{\frac{\partial^\beta f}{\partial w^\beta}(w\cdot e^{i\theta})} \,d\theta_n \dots d\theta_1 }\\& \leq \norm{\frac{\partial^\beta f}{\partial w^\beta}}_\infty \\&<\infty ,\end{align*}
 where the finiteness of the sup norm of the derivative follows since $f\in \mathcal{C}^\infty(\ol{\Omega})$ and $\Omega$ is assumed to be bounded. Therefore the function on $\Omega\setminus Z$ given by  $w\mapsto {c_{\alpha}}{w^{-\gamma}} $ is bounded and therefore extends holomorphically across the analytic set $Z$ to the function given by the same formula on $\Omega$, and the extended function admits the same bound. Since the origin is a boundary point of $\Omega$,  this means that if $\gamma\not=0$, then $c_\alpha=0$, so that the Laurent series in \eqref{eq-laurent} reduces to a {\em Taylor} series
 \begin{equation} \label{eq-taylor} \sum_{\alpha \in \N^n} c_\alpha z^\alpha\end{equation}
i.e. there are no terms with negative powers of the coordinates, and represents the function $f$ on $\Omega$.
 Let $w\in\Omega$ so that the series \eqref{eq-taylor} converges when $z=w$. It follows from Abel's lemma (\cite[Lemma 1.15]{range}) that the series  
\eqref{eq-taylor} actually  converges in the polydisc $P_w=\{z\in \cx^n|  \abs{z_1}<\abs{w_1},\dots, \abs{z_n}<\abs{w_n}\}$.  Therefore, \eqref{eq-taylor} converges 
on the open set $\Omega\cup P_w$ to a holomorphic function $\wt{f}$, and $\wt{f}|_\Omega=f$. The proof is complete, since $P_w$ does not depend in any way on the choice of the function $f$.  \end{proof}
Recall that Reinhardt domain in $\cx^n$ is {\em log-convex} if its image in $\rl^n$ under the map $z\mapsto (\log\abs{z_1},\dots, \log\abs{z_n})$ is convex. An examination  of the proof above shows the following more precise version of the theorem actually holds.
\begin{cor*}Let $\Omega$ be as in the theorem above. Then each function in $\mathcal{C}^\infty(\ol{\Omega})\cap\mathcal{O}(\Omega)$ extends holomorphically to the the 
smallest complete log-convex Reinhardt domain containing the domain $\Omega$. 
\end{cor*}
\begin{proof} The last part of the proof of the theorem shows that each function
in   $\mathcal{C}^\infty(\ol{\Omega})\cap\mathcal{O}(\Omega)$  extends holomorphically to the open set
\[ \wt{\Omega}=\bigcup_{w\in \Omega} P_w= \bigcup_{w\in \Omega} \left\{z\in \cx^n|  \abs{z_1}<\abs{w_1},\dots, \abs{z_n}<\abs{w_n}\right\},\]
which  is  the smallest {\em complete} Reinhardt domain containing the domain $\Omega$, and on $\wt{\Omega}$ the Taylor series representation \eqref{eq-taylor} holds.  By a well-known classical result
 (see \cite[Chapter 2, Theorem 3.28]{range}), the series \eqref{eq-taylor} in fact converges in the smallest log-convex complete Reinhardt domain containing $\wt{\Omega}$, thus defining a holomorphic extension.  The corollary follows.
\end{proof}

{\bf Acknowledgments:} The author would like to thank Peter Pflug for bringing to 
his notice the results of \cite{jarnickipflug1997}. He would also like to thank the editor, Harold P. Boas,  for his comments leading to significant improvements in the paper. 

\bibliographystyle{plain}
\bibliography{power}{}

\end{document}